\documentclass{amsart}
\usepackage{amsmath, amssymb, amsthm}
\usepackage[dvipdfmx]{graphicx}
\usepackage{mathrsfs}

\def \vmo{\operatorname{VMO}}
\def \bmo{\operatorname{BMO}}

\def \w{\omega}

\def \R{\mathbb{R}}

\newtheorem{theorem}{Theorem}
\newtheorem{lemma}[theorem]{Lemma}
\newtheorem{proposition}[theorem]{Proposition}

\newtheorem{claim}{Claim}

\theoremstyle{definition}

\newtheorem{remark}{Remark}

\numberwithin{equation}{section}

\begin{document}
\allowdisplaybreaks[4]
\title[A description of $A_{\infty}$-weights for VMO]{A description of $A_{\infty}$-weights for VMO}

\author{Jinsong Liu\textsuperscript{1,2} $\&$ Fei Tao\textsuperscript{1,2} $\&$ Huaying Wei\textsuperscript{3}}
\address{$1.$ HLM, Academy of Mathematics and Systems Science,
Chinese Academy of Sciences, Beijing, 100190, China}
\address{$2.$ School of
Mathematical Sciences, University of Chinese Academy of Sciences,
Beijing, 100049, China }
\address{$3.$ Department of Mathematics and Statistics, Jiangsu Normal University, Xuzhou, 221116, China}
\email{liujsong@math.ac.cn,\:ferrytau@amss.ac.cn,\:hywei@jsnu.edu.cn}
\thanks{Jinsong Liu is supported by National Key R$\&$D Program of China (Grant No. 2021YFA1003100), NSFC (Grant No. 11925107), Key Research Program of Frontier Sciences, CAS (Grant No. ZDBS-LY-7002); Huaying Wei is supported by the National Natural Science Foundation of China (Grant No. 12271218).}


\begin{abstract}
We present a new characterization of Muckenhoupt $A_{\infty}$-weights whose logarithm is in $\vmo(\R)$ in terms of vanishing Carleson measures on $\R_+^2$ and vanishing doubling weights on $\R$. This also gives a novel description of strongly symmetric homeomorphisms on the real line by using a geometric quantity.
\end{abstract}

\maketitle

\section{Introduction}
A locally integrable non-negative measurable function $\omega$ on $\mathbb R$ is called a \textit{weight}. We say that $\w$ is a \textit{doubling weight} if there exists a positive constant $\rho$ such that
$$
\rho^{-1}\w(J) \leq \w(I) \leq \rho \w(J)
$$
for any adjacent bounded intervals $I$ and $J$ of length $|I| = |J|$. Here, $\w(I) = \int_I\w(x) dx$. We call the optimal value of such $\rho$ the doubling constant for $\w$. Moreover, a doubling weight $\w$ is called \textit{vanishing} if $\w(I)/\w(J) \to 1$ as $|I| = |J| \to 0$. We say that $\w$ is a \textit{Muckenhoupt $A_{\infty}$-weight} (abbreviated to $A_{\infty}$-weight) (see \cite[C.6]{Ga}) if for any $\varepsilon > 0$ there exists some $\delta > 0$ such that
$$
|E| \leq \delta |I| \Rightarrow \w(E) \leq \varepsilon \w(I)
$$
whenever $I \subset \R$ is a bounded interval and $E\subset I$ a measurable subset. Naturally an $A_{\infty}$-weight is doubling. Fefferman and Muckenhoupt \cite{FM} gave this a direct computation, and they also provided an example of a function that satisfies the doubling condition but not $A_{\infty}$.

For a weight function $\w$ on the real line $\R$, we define a sense-preserving homeomorphism $h: \R \to \R$ by $h(x) = h(0) + \int_0^x \w(t) dt$. The homeomorphism $h$ is called \textit{strongly quasisymmetric} if $\w$ is an $A_{\infty}$-weight. In particular, $\log\w \in \bmo(\R)$, the space of functions of bounded mean oscillation on the real line (see Section 2 for precise definition). This subclass of quasisymmetric homeomorphisms and its Teichm\"uller space were much investigated (see \cite{AZ, CZ, FKP, Se}) because of their great importance in the application to harmonic analysis and elliptic operator theory (see \cite{Dah, FKP, Jo, Se1}). In particular, it was proved that a sense-preserving homeomorphism $h$ is strongly quasisymmetric if and only if it can be extended to a quasiconformal homeomorphism of $\R_+^2$ onto itself whose Beltrami coefficient $\mu$ induces a Carleson measure $|\mu(z)|^2y^{-1}dxdy$ on $\R_+^2$. Moreover, a strongly quasisymmetric homeomorphism $h$ is said to be \textit{strongly symmetric} if the $A_{\infty}$-weight $\w$ satisfies that $\log\w \in \vmo(\R)$, the space of functions of vanishing mean oscillation on the real line (see Section 2 for precise definition). This class was investigated further in \cite{WM, WM1}. In particular, it was proved that $h$ is strongly symmetric if and only if it can be extended to a quasiconformal homeomorphism of $\R_+^2$ onto itself whose Beltrami coefficient $\mu$ induces a vanishing Carleson measure $|\mu(z)|^2y^{-1}dxdy$ on $\R_+^2$.

Here, a positive measure $\lambda(x, y) d x d y$ on $\mathbb{R}_+^2$ is called a \textit{Carleson measure} if
\[
\|\lambda\|_{c}^{1 / 2}=\sup \frac{1}{|I|} \int_{0}^{|I|} \int_{I} \lambda(x, y) d x d y <\infty,
\]
where the supremum is taken over all bounded intervals $I \subset \mathbb R$. 
Furthermore, if a Carleson measure $\lambda$ satifies
$$\lim _{|I| \to 0} \frac{1}{|I|} \int_{0}^{|I|} \int_{I} \lambda(x, y) d x d y=0,$$
we call $\lambda$ a \textit{vanishing Carleson measure}.
%\begin{remark}
%The strongly quasisymmetric homeomorphisms and its Teichm\"uller space are usually defined and discussed on the unit circle $\mathbb S$. Due to the conformal invariance of $A_{\infty}$-weights and Carleson measures, we have the same theory on the real line $\R$. On the other hand, a sense-preserving homeomorphism $g$ on $\mathbb S$ is called \textit{strongly symmetric} if it is absolutely continuous such that $\log g'$ belongs to $\vmo(\mathbb S)$, the space of functions of vanishing mean oscillation on $\mathbb S$. This class was introduced in \cite{Pa} when Partyka studied eigenvalues of quasisymmetric automorphisms determined by VMO functions. It was investigated further later in \cite{SW, Wei} during their study of BMO theory of the universal Teichm\"uller space. Since the logarithmic derivative does not have conformal invariance, the notion of strongly symmetric homeomorphism on the real line is different from the one on the unit circle. In fact, the space of strongly symmetric homeomorphisms on $\mathbb S$ is smaller than that on $\R$ in the sense of conformal automorphism (see \cite{WM1}).
%\end{remark}

The purpose of the present paper is to give a new description of strongly symmetric homeomorphisms without using quasiconformal extensions (see Theorem \ref{main} below). Before stating the theorem we need to introduce some notations. For $z = (x, y) \in \mathbb R_+^2$,
let $I_z = \{s\mid \, |s - x| < y/2\}$, and let $I_z^-$ and $I_z^+$ be the left and right half of the interval $I_z$, respectively. For a doubling weight $\omega$ on $\mathbb R$, we introduce a geometric quantity
$$
\frac{2^{-1}\omega(I_z)}{\omega(I_z^{+})^{\frac{1}{2}}\omega(I_z^{-})^{\frac{1}{2}}},
$$
which is bigger or equal to $1$ since it is 
actually the ratio between the arithmetic mean and the geometric mean of the densities of $\omega$ over the intervals $I_z^{+}$ and $I_z^{-}$, that is $\omega\left(I_z^{+}\right) /\left|I_z^{+}\right|$ and $\omega\left(I_z^{-}\right) /\left|I_z^{-}\right|$. Then we define a nonnegative quantity 
$$
\eta (z) = \log\frac{2^{-1}\omega(I_z)}{\omega(I_z^{+})^{\frac{1}{2}}\omega(I_z^{-})^{\frac{1}{2}}}.
$$

In \cite{GN}, it was shown that the doubling weight $\w$ is an $A_{\infty}$-weight if and only if the positive measure $\eta(z)y^{-1} dxdy$ is a Carleson measure on $\mathbb R_{+}^{2}$. This expresses the close connection between $A_{\infty}$-weights (or strongly quasisymmetric homeomorphisms) and Carleson measures. We consider to what extent one can extend this result to $A_{\infty}$-weights whose logarithm is in $\vmo(\R)$ (or strongly symmetric homeomorphisms) and prove the following.

\begin{theorem}\label{main}
For an $A_{\infty}$-weight $\omega$ on $\mathbb R$, the function $\log\omega \in \rm VMO(\mathbb R)$ if and only if the Carleson measure $\eta(z)y^{-1} dxdy$ is vanishing on $\mathbb R_{+}^{2}$ and the doubling weight $\omega$ is vanishing on $\mathbb R$.
\end{theorem}

\begin{remark}\label{examin}
The condition that $\w$ is an $A_{\infty}$-weight with $\log\w \in \vmo(\R)$ implies that the doubling weight $\omega$ is vanishing on $\R$, which we may obtain by examining the proof of corresponding claim for the case of unit circle in \cite[Lemma 3.3]{Sh18}.
\end{remark}

\begin{remark}
Let $\tilde\eta (z) = \left|1 - \omega(I_z^{+})/\omega(I_z^{-})\right|^2.$ Since the weight $\w$ is doubling, it is easy to compute that $\tilde\eta (z)$ and $\eta (z)$ are comparable with comparison constant depending only on doubling constant for $\w$. Then, the Carleson measure $\eta(z)y^{-1} dxdy$ is vanishing on $\mathbb R_{+}^{2}$ if and only if the Carleson measure $\tilde\eta(z)y^{-1} dxdy$ is. We consider $\lambda_{\delta}(\w) = \sup_{0 < y \leq \delta}\tilde \eta(z)$. Then the doubling weight $\w$ is vanishing on $\R$ if and only if $\lim_{\delta \to 0^+} \lambda_{\delta}(\w) = 0$. Further, if the rate of convergence of $\lambda_{\delta}(\w)$ satisfies the condition
\begin{equation*}\label{lambda}
\int_0 \frac{\lambda_{\delta}(\w)}{\delta} d\delta < \infty,
\end{equation*}
 then, by the estimate 
\begin{align*}
 \int_{I(x_0, t)} \int_0^t\tilde\eta (z) \frac{dxdy}{y} \leq \int_{I(x_0, t)} \int_0^t\lambda_{y}(\w) \frac{dxdy}{y} = t\int_0^t \frac{\lambda_y(\w)}{y}dy, 
\end{align*}
we can conclude that the measure $\tilde\eta(z)y^{-1} dxdy$ is a vanishing Carleson measure on $\mathbb R_{+}^{2}$.
\end{remark}

The paper is structured as follows: in Section 2, we give some basic definitions and results on BMO functions and Muckenhoupt weights which will be used in the proof of Theorem \ref{main}. Section 3 is devoted to the proof of Theorem \ref{main}. 
%In the final Appendix section, we provide an alternative explanation on the vanishing Carleson measure in Theorem \ref{main}. 

\section{Preliminaries}
Let $I_{0}$ be any interval on the real line $\R$. A locally integrable function $u \in L^{1}_{loc}\left(I_{0}\right)$ is said to have \textit{bounded mean oscillation} (abbreviated to BMO) if
\[
\|u\|_{\bmo\left(I_{0}\right)} = \sup \frac{1}{|I|} \int_{I}\left|u(x)-u_{I}\right| d x <\infty,
\]
where the supremum is taken over all bounded intervals $I$ of $I_{0}$, $u_{I}$ is the average of $u$ on the interval $I$. The set of all BMO functions on $I_0$ is denoted by ${\rm BMO}(I_0)$. This is regarded as a Banach space with norm $\Vert \cdot \Vert_{\bmo\left(I_{0}\right)}$ modulo constants since obviously constant functions have norm zero. Moreover, if $u$ also satisfies the condition
\[
\lim _{|I| \rightarrow 0} \frac{1}{|I|} \int_{I}\left|u(x)-u_{I}\right|d x= 0,
\]
we say $u$ has \textit{vanishing mean oscillation} (abbreviated to VMO). The set of all VMO functions on $I_0$ is denoted by ${\rm VMO}(I_0)$.
This is a closed subspace of ${\rm BMO}(I_0)$. 
%Functions of bounded mean oscillation in $\mathbb{R}^n$ were first introduced by John and Nirenberg in \cite{JN} and they applied them to smoothness problems in partial differential equation. 
The {\it John--Nirenberg inequality} for BMO functions (see \cite[VI.2]{Ga}, \cite[IV.1.3]{St}) asserts that
there exists two universal positive constants $C_1$ and $C_2$ such that for any $u \in {\rm BMO}(I_0)$,
any bounded interval $I$ of $I_0$, and any $\lambda > 0$, it holds that
\begin{equation}\label{JN}
\frac{1}{|I|} |\{t \in I: |u(t) - u_I| \geq \lambda \}| \leq C_1 \exp\left(\frac{-C_{2}\lambda}{\Vert u \Vert_{\bmo(I_0)}} \right).
\end{equation}

We say that the weight $\omega$ is
a {\it Muckenhoupt $A_p$-weight} (abbreviated to $A_p$-weight) for $p>1$ if there exists a constant $C_p(\omega) \geq 1$ such that
\begin{equation}\label{Ap}
\left(\frac{1}{|I|} \int_I \omega(x)dx \right)\left(\frac{1}{|I|} \int_{I} \left(\frac{1}{\omega(x)}\right)^{\frac{1}{p-1}}dx\right)^{p-1}
\leq C_p(\omega)
\end{equation}
for any bounded interval $I \subset \mathbb R$. We call the optimal value of such $C_p(\omega)$
the $A_p$-constant for $\omega$. It is known that $\bigcup_{p > 1} A_p = A_\infty$ and $A_p \subset A_q$ for $p < q$ (see \cite{M}).

The Jensen inequality implies that
\begin{equation}\label{Jensen}
\exp \left(\frac{1}{|I|} \int_I \log \omega(x) dx \right)
\leq \frac{1}{|I|} \int_I \omega(x) dx.
\end{equation}
Another characterization of $A_\infty$-weights can be given by the reverse Jensen inequality. Namely, $\omega \geq 0$ belongs to the class of $A_\infty$-weights
if and only if there exists a constant $C_\infty(\omega) \geq 1$ such that
\begin{equation}\label{iff}
\frac{1}{|I|} \int_I \omega(x) dx \leq C_\infty(\omega) \exp \left(\frac{1}{|I|} \int_I \log \omega(x) dx \right)
\end{equation}
for every bounded interval $I \subset \mathbb R$ (see \cite{Hr}).
We call the optimal value of such $C_\infty(\omega)$
the $A_\infty$-constant for $\omega$.

Sarason gave a characterization of $\vmo$ functions by means of $A_2$-weights (see \cite[Theorem 2]{Sa}).
\begin{proposition}\label{sarason}
Let $\w$ be a weight function with $\log \w \in \bmo(\R)$. Then, $\log \w$ belongs to $\vmo(\R)$ if and only if
\begin{align}\label{sara}
\lim_{|I| \to 0}\left(\frac{1}{|I|}\int_{I}\w(x) d x\right)\left(\frac{1}{|I|}\int_{I}\frac{1}{\w(x)}d x\right)=1.
\end{align}
\end{proposition}
Here, \eqref{sara} may be thought of as a limit $A_2$-condition.
Inspired by Proposition \ref{sarason}, Mitsis \cite{Mi} pushed the analogy between $A_p$-weights and $A_{\infty}$-weights further by replacing the limit $A_2$-condition with the following so-called asymptotic reverse Jensen inequality (see \eqref{equiv} below):
\begin{proposition}\label{limit}
Let $\omega$ be a weight function with $\log \omega \in \bmo(\R)$. Then, $\log \omega$ belongs to $\vmo(\R)$ if and only if
\begin{align}\label{equiv}
\lim_{|I| \to 0}\left(\frac{1}{|I|}\int_{I}\omega(x)d x\right)\exp{\left(-\frac{1}{|I|}\int_{I}\log \omega(x) d x\right)}=1.
\end{align}
In other words,
\begin{align}\label{equiv1}
\lim_{|I| \to 0}\left(\log \omega_I - (\log\omega)_I \right) = 0.
\end{align}
\end{proposition}
More precisely, Mitsis \cite{Mi} proved sufficiency with respect to nonatomic measures (covering Lebesgue measure) by performing a dyadic decomposition of the involved interval and omitted the detailed proof of necessity by pointing out it is a standard argument involving the John-Nirenberg inequality for nonatomic measures. Proposition \ref{limit} is relevant to the proof of Theorem \ref{main} in the following section. For the completeness of our paper, we complement the proof of necessity and give sufficiency a simple proof involving only elementary measure theoretic considerations.

\begin{remark}
Put it differently, Propositions \ref{sarason} and \ref{limit} imply that $A_2$-condition and $A_{\infty}$-condition coincide if one restricts to weights which tend to be constant on arbitrarily small intervals.
\end{remark}

\begin{proof}[Proof of Proposition \ref{limit}]
Suppose $\log\omega \in \vmo(\R)$. 
%Then, for any $\varepsilon > 0$, there exists $\delta > 0$ such that for any bounded interval $I_0 \subset \R$ with $|I_0| \leq \delta$ we have $\Vert u \Vert_{\bmo(I_0)} < \min\{C_2\varepsilon/(2C_1), C_2/2\}$. Here, $C_1$ and $C_2$ are constants in \eqref{JN}. Then, for any bounded interval $I \subset I_0$, the John-Nirenberg inequality yields that
%\begin{equation}\label{applyJN}
%\begin{split}
%\frac{1}{|I|} \int_{I} e^{\left|u(x)-u_{I}\right|} dx &=\frac{1}{|I|} \int_{0}^{\infty}\left|\left\{x \in I:\left|u(x)-u_{I}\right| \geq \lambda\right\}\right| e^{\lambda} d\lambda + 1 \\
%& \leq C_{1} \int_{0}^{\infty} \exp \left(\frac{-C_{2} \lambda}{\|u\|_{\bmo\left(I_{0}\right)}}\right) e^{\lambda}d\lambda + 1\\
%& = \frac{C_{1}\|u\|_{\bmo\left(I_{0}\right)}}{C_{2}-\|u\|_{\bmo\left(I_{0}\right)}} + 1 < \varepsilon + 1.
%\end{split}
%\end{equation}
By the Jensen inequality, we see
$$
1\leq \left(\frac{1}{|I|}\int_{I}\omega(x)d x\right)\exp{\left(-\frac{1}{|I|}\int_{I}\log \omega(x) d x\right)}
\leq \left(\frac{1}{|I|}\int_{I}\omega(x)dx\right)\left(\frac{1}{|I|}\int_{I}\frac{1}{\omega(x)}dx\right).
$$
Then, obviously, \eqref{equiv} follows from \eqref{sara}.
%and moreover, by \eqref{applyJN} and $ab \leq \frac{1}{4}(a+b)^2$ for $a, b \geq 0$, we have
%\begin{align*}
%\left(\frac{1}{|I|}\int_{I}\omega(x)dx\right)\left(\frac{1}{|I|}\int_{I}\frac{1}{\omega(x)}dx\right)
%& = \left(\frac{1}{|I|}\int_{I}e^{u(x) - u_I}dx\right)\left(\frac{1}{|I|}\int_{I}e^{u_I - u(x)} dx\right)\\
%& \leq \frac{1}{4}\left(\frac{1}{|I|}\int_{I}e^{u(x) - u_I}dx + \frac{1}{|I|}\int_{I}e^{u_I - u(x)} dx \right)^2\\
%& = \frac{1}{4}\left(\frac{1}{|I|}\int_{I}e^{|u(x) - u_I|}dx + \frac{1}{|I|}\int_{I}e^{-|u(x) - u_I|} dx \right)^2\\
%& \leq \frac{1}{4}\left(\frac{1}{|I|}\int_{I}e^{|u(x) - u_I|}dx + 1 \right)^2 \leq \frac{1}{4}(\varepsilon + 2)^2.
%\end{align*}
%Thus, \eqref{equiv} is satisfied. This gives the proof of necessity.

Suppose \eqref{equiv} holds. To show $u := \log\omega \in \vmo(\R)$, we use a strategy of measure theory in \cite[Lemma 3]{Sa}.
Let $I$ be a bounded interval in $\R$ such that
\begin{align}\label{equation}
\left(\frac{1}{|I|}\int_{I}\omega(x)d x\right)\exp{\left(-\frac{1}{|I|}\int_{I}\log \omega(x) d x\right)} = 1 + \varepsilon^{3}
\end{align}
for $0 < \varepsilon < 1/2$. Assuming
\begin{align}\label{assume}
u_I = \frac{1}{|I|}\int_{I}\log \w(x) d x=0, 
\end{align}
we have
$$\frac{1}{|I|}\int_{I}\w(x)d x =1 + \varepsilon^3.$$
Let $F$ be the set where $e^{-\varepsilon}<\w<e^{\varepsilon}$ and $E=I-F$. We have
\begin{align*}
	(1 + \varepsilon^3)|I|&=\int_{E}(\w(x)-\log\w(x))dx+\int_{F}(\w(x)-\log\w(x))dx\\
	&\geq (e^{-\varepsilon}+\varepsilon)|E|+|F|\\
	&\geq (1+\frac{1}{4}\varepsilon^{2})|E|+|F|\\
	&=|I|+\frac{1}{4}\varepsilon^{2}|E|,
\end{align*}
which implies $|E|\leq 4\varepsilon|I|$ and $|F|\geq (1-4\varepsilon)|I|$. Thus,
\begin{align*}
	\int_{E}\w(x)d x &=(1 + \varepsilon^3)|I|-\int_{F}\w(x)d x\\
	&\leq (1+\varepsilon)|I|-e^{-\varepsilon}|F|\\
	&\leq (1+\varepsilon)|I|-(1-\varepsilon)(1-4\varepsilon)|I|\\
	&<6\varepsilon|I|.
\end{align*}
On the other hand, by \eqref{assume} we have
$$-\int_{E}\log \w(x) d x=\int_{F}\log \w(x) d x \leq \varepsilon|F|\leq \varepsilon|I|.$$

Noting that $|\log \w| < \varepsilon$ on $F$, and $|\log \w | \leq \w -\log\w$ generally, we conclude that
\begin{align*}
	\int_{I}|\log\w(x)| d x&=\int_{E}|\log\w(x)| d x+\int_{F}|\log\w(x)| d x\\
	&\leq \int_{E}(\w(x)-\log\w(x))dx+\varepsilon|F|\\
	&\leq 6\varepsilon |I| + \varepsilon |I| + \varepsilon |I| = 8\varepsilon |I|.
\end{align*}
Combined with \eqref{assume}, this implies that
\begin{align}\label{55}
\frac{1}{|I|}\int_{I}\left|u(x)- u_I\right|d x < 8 \varepsilon
\end{align}
for $u = \log\w$. If $\log\w$ does not satisfy \eqref{assume}, then we write $\log\w = (\log\w - a) + a$ with $a = \frac{1}{|I|}\int_{I}\log\w(x) dx$. Since $\log\w - a$ satisfies \eqref{equation}, and \eqref{assume} holds for $\log\w - a$, we conclude from \eqref{55} that
\begin{align*}
\frac{1}{|I|}\int_{I}\left|u(x)- u_I\right|d x = \frac{1}{|I|}\int_{I}\left|(u(x) - a)- (u - a)_I\right|d x < 8 \varepsilon.
\end{align*}
Thus, \eqref{55} holds for every $u = \log\w$ which satisfies \eqref{equation}. Consequently, \eqref{equiv} implies $u = \log\w \in \vmo(\R)$. This completes the proof of Proposition \ref{limit}.
\end{proof}

\section{Proof of Theorem 1}
In this section, we focus on the proof of Theorem \ref{main}.

For any $x_0 \in \mathbb R$ and $t > 0$, we set $I(x_0, t) = \{x\mid \, |x - x_0| < t/2\}$ and set
$$
A(x_0, t) = \frac{1}{t} \int_0^t\int_{I(x_0, t)} \eta (z) \frac{dxdy}{y}=\frac{1}{t}\int_{I(x_0, t)} \int_0^t \eta (z) \frac{dydx}{y}.
$$
It is remarkable that $\eta(z)/y$ is locally integrable (see \cite{GN}), the above equality holds due to Fubini's theorem.
 
Considering that the integrand $\eta (z)/y$ is nonnegative, we divide the integral by $dy$ over $[0, t]$ into those on dyadic intervals and then by changing the variables, we obtain
\begin{align*}
\int_0^t \int_{I(x_0, t)} \eta (z) \frac{dy}{y} &= \sum_{k=0}^{\infty} \int_{\frac{t}{2^{k+1}}}^{\frac{t}{2^k}}\int_{I(x_0, t)} \log\frac{2^{-1} \omega(I_z)}{\omega(I_z^+)^{\frac{1}{2}}\omega(I_z^-)^{\frac{1}{2}}}\frac{dy}{y}\\
&= \lim_{N\to\infty}\sum_{k=1}^{N}\int_{\frac{t}{2}}^{t}\int_{I(x_0, t)} \log\frac{2^{-1}\omega(x - \frac{y}{2^k}, x+\frac{y}{2^k})}{ \omega(x - \frac{y}{2^k}, x)^{\frac{1}{2}} \omega(x, x+\frac{y}{2^k})^{\frac{1}{2}}} \frac{dy}{y}. 
\end{align*}
By rearranging the order of the following sum:
\begin{align}\label{ps}
&\quad \sum_{k=1}^{N}\log\frac{2^{-1}\omega(x - \frac{y}{2^k}, x+\frac{y}{2^k})}{ \omega(x - \frac{y}{2^k}, x)^{\frac{1}{2}} \omega(x, x+\frac{y}{2^k})^{\frac{1}{2}}} \\\notag
&= \log\frac{2^{-N}\omega(x - \frac{y}{2}, x+\frac{y}{2})}{ \omega(x - \frac{y}{2^N}, x)^{\frac{1}{2}} \omega(x, x+\frac{y}{2^N})^{\frac{1}{2}}} + \sum_{k=1}^{N-1}\log\frac{\omega(x - \frac{y}{2^{k+1}}, x+\frac{y}{2^{k+1}})}{ \omega(x - \frac{y}{2^k}, x)^{\frac{1}{2}} \omega(x, x+\frac{y}{2^k})^{\frac{1}{2}}}
\end{align}	
and by using further observation on the ratio of the first term:
\begin{align}\label{partialsum}
	\frac{2^{-N}\omega(x - \frac{y}{2}, x+\frac{y}{2})}{ \omega(x - \frac{y}{2^N}, x)^{\frac{1}{2}} \omega(x, x+\frac{y}{2^N})^{\frac{1}{2}}} = \frac{\omega(x - \frac{y}{2}, x+\frac{y}{2})/y}{ (\omega(x - \frac{y}{2^N}, x)/\frac{y}{2^N})^{\frac{1}{2}} (\omega(x, x+\frac{y}{2^N})/\frac{y}{2^N})^{\frac{1}{2}}},
\end{align}
from \eqref{ps} and \eqref{partialsum} we see that $N$-th partial sum of the series $A(x_0, t)$ can be written as
\begin{align*}
&\sum_{k=1}^{N} \frac{1}{t}\int_{\frac{t}{2}}^{t}\int_{I(x_0, t)} \log\frac{2^{-1}\omega(x - \frac{y}{2^k}, x+\frac{y}{2^k})}{ \omega(x - \frac{y}{2^k}, x)^{\frac{1}{2}} \omega(x, x+\frac{y}{2^k})^{\frac{1}{2}}} \frac{dxdy}{y}\\
& = \frac{1}{t}\int_{\frac{t}{2}}^{t}\int_{I(x_0, t)}\log	\frac{\omega(x - \frac{y}{2}, x+\frac{y}{2})}{y} \frac{dxdy}{y} \\
&\quad - \frac{1}{2t} \left( \int_{\frac{t}{2}}^{t}\int_{I(x_0, t)}\log \frac{\omega(x - \frac{y}{2^N}, x)}{\frac{y}{2^N}} \frac{dxdy}{y} + \int_{\frac{t}{2}}^{t}\int_{I(x_0, t)}\log \frac{\omega(x, x+\frac{y}{2^N})}{\frac{y}{2^N}} \frac{dxdy}{y}\right)\\
&\quad + \sum_{k=1}^{N-1}\frac{1}{t}\int_{\frac{t}{2}}^{t}\int_{I(x_0, t)}\log\frac{\omega(x - \frac{y}{2^{k+1}}, x+\frac{y}{2^{k+1}})}{ \omega(x - \frac{y}{2^k}, x)^{\frac{1}{2}} \omega(x, x+\frac{y}{2^k})^{\frac{1}{2}}} \frac{dxdy}{y}\\
 &:= A_1(x_0,t)- A_2(x_0,t) + A_3(x_0,t).
\end{align*}
Then, we can separate $A(x_0, t)$ into four parts as follows: 
 \begin{align}\label{3terms}
A(x_0, t) =& \lim_{N\to \infty}[A_1(x_0,t)- A_2(x_0,t) + A_3(x_0,t)]\notag\\
 =& (A_1(x_0,t) - \log 2\log\omega_{I(x_0, t)}) -\lim_{N\to\infty} [A_2(x_0,t) - \log 2 (\log\omega)_{I(x_0, t)}] \\\notag
&+ \lim_{N\to\infty} A_3(x_0,t)+ \log 2 (\log \omega_{I(x_0, t)} - (\log\omega)_{I(x_0, t)})\\\notag
:=& \widehat{A_1}(x_0,t)- \lim_{N\to\infty}\widehat{A_2}(x_0,t) +\widehat{A_3}(x_0,t)+A_4.
\end{align}
Here, $\w_I$ is the average of $\w$ on the interval $I$ as above; namely, $\w_I = \frac{1}{|I|}\int_I\w(x)dx$.

\begin{claim}\label{c1}
If $\omega$ is an $A_\infty$-weight on $\mathbb{R}$, it holds that
$$\lim_{N\to \infty} \widehat{A_2}(x_0,t) = 0.$$
\end{claim}
\begin{proof}
We shall apply Lebesgue's dominated theorem to prove the claim. For any fixed $x_0,\ t$, set $h(x)=\log \omega(x)\chi_{I(x_0,2t)}$. Then using our hypothesis on $\omega$ that $\omega$ is an $A_\infty$-weight, it readily follows that
\begin{align*}
	&\left|\log\frac{\omega(x-\frac{y}{2^N},x)}{\frac{y}{2^{N}}}\right|\\
\leq&\left|\log\frac{\omega(x-\frac{y}{2^N},x)}{\frac{y}{2^{N}}}-\frac{1}{\frac{y}{2^{N}}}\int_{x-\frac{y}{2^{N}}}^{x}\log\omega(u)du\right|+\frac{1}{\frac{y}{2^{N}}}\int_{x-\frac{y}{2^{N}}}^{x}\left|\log\omega(u)\right|du\\
\leq&C+Mh(x),
\end{align*}
where $M$ is Hardy-Littlewood maximal operator. It is not hard to see that $\log\omega\in \operatorname{BMO}(\mathbb{R})$ and hence $h(x)\in L^{1}(\mathbb{R})$. Therefore, $Mh$ is in weak-$L^1$ and thus locally integrable. From Lebesgue's dominated theorem and Lebesgue differentiation Theorem, we deduce that
$$\lim_{N\to\infty}\frac{1}{t}\int_{\frac{t}{2}}^{t}\int_{I(x_0,t)}\log\frac{\omega(x-\frac{y}{2^N},x)}{\frac{y}{2^{N}}}\frac{dxdy}{y}=\log2 (\log\omega)_{I(x_0,t)}.$$
This concludes the proof of Claim \ref{c1}.
\end{proof}
	
We give a simple observation on vanishing doubling weights, which will be used to prove Claims \ref{c2} and \ref{c3} in the following. 
\begin{lemma}\label{vanish}
Let the doubling weight $\omega$ be vanishing; namely, for any $\varepsilon > 0$ there exists some $\delta > 0$ such that
\begin{equation}\label{inequ}
(1 + \varepsilon)^{-1} \leq \frac{\omega(x_0, x_0 + t)}{\omega(x_0 - t, x_0)} \leq 1 + \varepsilon
\end{equation}
for every $x_0 \in \mathbb R$ and for every $t \in (0, \delta)$. Then,\\\

\noindent$\begin{aligned}
&(a)\;\frac{(1 + \varepsilon)^n - 1}{\varepsilon(1 + \varepsilon)^{n-1}} \leq \frac{\omega(x_0 - nt, x_0 + nt)}{\omega(x_0 - t, x_0 + t)} \leq \frac{(1 + \varepsilon)^n - 1}{\varepsilon},\\
&(b)\;(1 + \varepsilon)^{-1} \leq \frac{\omega(x_1, x_1 + t)}{\omega(x_0 - t, x_0)} \leq 1 + \varepsilon
\end{aligned}$\\
are satisfied for every positive integer $n$ and for any $x_1 \in (x_0 -t, x_0)$, respectively.
\end{lemma}	

\begin{proof}	
We now estimate $\omega(x_0 - nt, x_0 + nt)$	from both above and below in terms of $\omega(x_0 - t, x_0 + t)$.
\begin{align*}
&\quad\omega(x_0 - nt, x_0 + nt)\\
& = \sum\nolimits_{k=1}^n\left(\omega(x_0-kt,x_0-(k-1)t)+\omega(x_0+(k-1)t,x_0+kt)\right)\\
& \leq \sum\nolimits_{k=1}^n(1+\varepsilon)^{k-1}\left(\omega(x_0-t,x_0)+\omega(x_0,x_0+t)\right)\\
& = \frac{(1 + \varepsilon)^n - 1}{\varepsilon}\omega(x_0 - t, x_0 + t),
\end{align*}
and similarly,
\begin{align*}
&\quad\omega(x_0 - nt, x_0 + nt)\\
& \geq \sum\nolimits_{k=1}^{n}((1+\varepsilon)^{-1})^{k-1}\omega(x_0 - t, x_0 + t)\\
& = \frac{(1 + \varepsilon)^n - 1}{\varepsilon(1 + \varepsilon)^{n-1}}\omega(x_0 - t, x_0 + t),
\end{align*}
from which we obtain the statement $(a)$.

To prove the statement $(b)$ we use the property of the mediant. It says that assuming
$$
(1+\varepsilon)^{-1} \leq \frac{b}{a},\;\frac{d}{c}\leq 1+\varepsilon
$$
we have
$$
(1+\varepsilon)^{-1} \leq \frac{b+d}{a+c} \leq 1+\varepsilon.
$$
Since
$$
(1+\varepsilon)^{-1} \leq \frac{\omega((x_0+x_1)/2, x_1+t)}{\omega(x_0-t, (x_0+x_1)/2)},\;\frac{\omega(x_1, (x_0+x_1)/2)}{\omega((x_0+x_1)/2, x_0)}\leq 1+\varepsilon,
$$
by the property of the mediant, the statement $(b)$ is proved.

\end{proof}

\begin{claim}\label{c2}
If the doubling weight $\omega$ is vanishing on $\mathbb R$, then it holds that
$$\lim_{t \to 0} \widehat{A_1}(x_0,t) = 0$$
uniformly for $x_0 \in \mathbb R$.
\end{claim}	

\begin{proof}
We shall make an estimate on $\widehat{A_1}(x_0,t)$:
\begin{equation}\label{A1}
\begin{split}
|\widehat{A_1}(x_0,t)| & \leq \frac{1}{t} \int_{I(x_0, t)}\int_{\frac{t}{2}}^{t} \left|\log\frac{\omega_{I(x, y)}}{\omega_{I(x_0, t)}}\right| \frac{dydx}{y} \\
& \leq \frac{1}{t} \int_{I(x_0, t)}\int_{\frac{t}{2}}^{t} \left|\log\frac{\omega_{I(x, y)}}{\omega_{I(x, t)}}\right| \frac{dydx}{y} + \frac{1}{t} \int_{I(x_0, t)}\int_{\frac{t}{2}}^{t} \left|\log\frac{\omega_{I(x, t)}}{\omega_{I(x_0, t)}}\right| \frac{dydx}{y}.
\end{split}
\end{equation}
Since $\omega$ is vanishing, for any arbitrarily small $\varepsilon > 0$ there exists some $\delta > 0$ such that \eqref{inequ} holds for every $x_0 \in \mathbb R$ and for every $t \in (0, \delta)$. We suppose $t \in (0, \delta)$ in the following.

First, we estimate the first term in the last line of \eqref{A1}. Set $N = 1/\sqrt\varepsilon$ (we may adjust $\varepsilon$ so that $N$ becomes an integer). Then,
$$
\int_{\frac{t}{2}}^{t} \left|\log\frac{\omega_{I(x, y)}}{\omega_{I(x, t)}}\right| \frac{dy}{y}
 \leq \sum_{k=1}^{N} \int_{\frac{t}{2}(1+\frac{k -1}{N})}^{\frac{t}{2}(1+\frac{k}{N})} |\log\omega_{I(x, y)} - \log\omega_{I(x, t)}| \frac{dy}{y}.
$$
We note that as $y \in [\frac{t}{2}(1+\frac{k-1}{N}), \frac{t}{2}(1+\frac{k}{N})]$, the difference value $|\log\omega_{I(x, y)} - \log\omega_{I(x, t)}|$ is less than the maximum of 
\begin{equation}\label{big}
\max_{1\leq k\leq N}\left|\log\left(\frac{1}{\frac{t}{2}(1+\frac{k-1}{N})} \int_{x-\frac{t}{4}(1+\frac{k}{N})}^{x+\frac{t}{4}(1+\frac{k}{N})}\omega(u)du\right) - \log\omega_{I(x,t)}\right|
\end{equation}
and
\begin{equation}\label{small}
\max_{1\leq k\leq N}\left|\log\left(\frac{1}{\frac{t}{2}(1+\frac{k}{N})} \int_{x-\frac{t}{4}(1+\frac{k-1}{N})}^{x+\frac{t}{4}(1+\frac{k-1}{N})}\omega(u)du\right) - \log\omega_{I(x,t)}\right|.
\end{equation}
We assume that \eqref{big} is bigger than \eqref{small} and continue our computation (the other case can be treated similarly). Combined with 
\begin{equation*}
	\sum_{k =1}^{N} \int_{\frac{t}{2}(1+\frac{k -1}{N})}^{\frac{t}{2}(1+\frac{k}{N})}\frac{dy}{y} = \log 2 < 1,
\end{equation*}
it implies 
\begin{equation*}
\begin{split}
\quad\int_{\frac{t}{2}}^{t}\left| \log\frac{\omega_{I(x, y)}}{\omega_{I(x, t)}}\right| \frac{dy}{y}\leq&\max_{1\leq k\leq N}\left|\log\left(\frac{1}{\frac{t}{2}(1+\frac{k-1}{N})} \int_{x-\frac{t}{4}(1+\frac{k}{N})}^{x+\frac{t}{4}(1+\frac{k}{N})}\omega(u)du\right) - \log\omega_{I(x,t)}\right|\\
\leq &\max_{1\leq k\leq N}\left\{\log\frac{\frac{t}{2}(1+\frac{k}{N})}{\frac{t}{2}(1+\frac{k-1}{N})} +\left|\log \omega_{I(x,\frac{t(N+k)}{2N})} - \log\omega_{I(x,t)}\right|\right\}\\
\leq &\log\left(1+\frac{1}{N}\right)+\max_{1\leq k\leq N}\left|\log \omega_{I(x,\frac{t}{2N}\times(N+k))} - \log\omega_{I(x,\frac{t}{2N}\times2N)}\right|.
\end{split}
\end{equation*}
By the statement $(a)$ in Lemma \ref{vanish}, we conclude 
\begin{equation*}
\begin{split}
&\left|\log \omega_{I(x,\frac{t}{2N}\times(N+k))} - \log\omega_{I(x,\frac{t}{2N}\times2N)}\right|\\
 \leq& \left|\log\left(\frac{(1+\varepsilon)^{N+k} - 1}{(N+k)\varepsilon}\omega_{I(x, \frac{t}{2N})}\right) - \log\left(\frac{(1+\varepsilon)^{2N} - 1}{2N\varepsilon(1+\varepsilon)^{2N-1}}\omega_{I(x, \frac{t}{2N})}\right)\right|\\
&+\left|\log\left(\frac{(1+\varepsilon)^{2N} - 1}{2N\varepsilon}\omega_{I(x, \frac{t}{2N})}\right) - \log\left(\frac{(1+\varepsilon)^{N+k} - 1}{(N+k)\varepsilon(1+\varepsilon)^{N+k-1}}\omega_{I(x, \frac{t}{2N})}\right)\right|\\
\leq&2\left|\log\left(\frac{(1+\varepsilon)^{N+k} - 1}{(1+\varepsilon)^{2N} - 1}\cdot \frac{2N}{N+k}\right)\right|+(3N+k-2)\log(1+\varepsilon).
\end{split}
\end{equation*}
The monotonicity of $\frac{a^x-1}{x}(a>1)$ yields that 
 \begin{equation}
 	\begin{split}
 		\int_{\frac{t}{2}}^{t} \left|\log\frac{\omega_{I(x, y)}}{\omega_{I(x, t)}} \right|\frac{dy}{y}\leq & \log\left(1+\frac{1}{N}\right)+2\left|\log\frac{2((1+\varepsilon)^{N} - 1)}{(1+\varepsilon)^{2N} - 1}\right|+(4N-2)\log(1+\varepsilon).
 	\end{split}
 \end{equation}
This can be arbitrarily small as $\varepsilon$ is sufficiently small. Therefore, the first term in the last line of \eqref{A1} tends to $0$ uniformly for $x_0 \in \mathbb R$ as $t\to 0$.

Next, we consider the second term in the last line of \eqref{A1}. By using the statement $(b)$ in Lemma \ref{vanish}, we see that
\begin{align*}
&\quad\frac{1}{t} \int_{I(x_0, t)}\int_{\frac{t}{2}}^{t} \left|\log\frac{\omega_{I(x, t)}}{\omega_{I(x_0, t)}}\right| \frac{dydx}{y}\\
& \leq \frac{1}{t} \int_{I(x_0, t)} \left|\log\frac{\omega_{I(x, t)}}{\omega_{I(x_0, t)}}\right| dx \times \log 2\\
& = \frac{1}{t} \int_{I(x_0, t)}\left| \log\frac{\omega(x - \frac{t}{2}, x+\frac{t}{2})}{\omega (x_0-\frac{t}{2}, x_0+\frac{t}{2})}\right| dx \times \log 2\\
&\leq \varepsilon\times\log 2.
\end{align*}
Thus, the second term in the last line of \eqref{A1} is bounded by a constant multiple of $\varepsilon$. This completes the proof of Claim \ref{c2}.
\end{proof}

\begin{claim}\label{c3}
If the doubling weight $\omega$ is vanishing on $\mathbb R$, then it holds that
$$
\lim_{t \to 0} \widehat{A_3}(x_0, t) = 0
$$
uniformly for $x_0 \in \mathbb R$.
\end{claim}

\begin{proof}
Set
\begin{equation}\label{Fk}
F_k(y) = \int_{I(x_0, t)}\log\frac{\omega(x - \frac{y}{2^{k+1}}, x+\frac{y}{2^{k+1}})}{ \omega(x - \frac{y}{2^k}, x)^{\frac{1}{2}} \omega(x, x+\frac{y}{2^k})^{\frac{1}{2}}} dx.
\end{equation}	
Then,
\begin{equation}\label{3}
 \widehat{A_3}(x_0, t) = \sum_{k=1}^{\infty}\frac{1}{t}\int_{\frac{t}{2}}^{t}F_k(y) \frac{dy}{y}.
\end{equation}
Assume that $\omega$ is vanishing; namely, for any arbitrarily small $\varepsilon > 0$ there exists some $\delta > 0$ such that \eqref{inequ} holds for every $x_0 \in \mathbb R$ and for every $t \in (0, \delta)$. We suppose $t \in (0, \delta)$ in the following. 

We now estimate the ratio $|F_k(y)/y|$. The expression (\ref{Fk}) for $F_k(y)$ is changed in form step by step in the following for convenience of the estimate. We first notice that
\begin{align*}
2F_k(y) = &2\int_{I(x_0, t)}\log\omega(I(x, \frac{y}{2^k})) dx - \int_{I(x_0, t)}\log\omega(x - \frac{y}{2^k}, x) dx\\
& - \int_{I(x_0, t)}\log\omega(x, x+\frac{y}{2^k}) dx.
\end{align*}
By the change of the variables, we make the integrands in the second and third terms same as the first one; namely,
\begin{align*}
&\quad \int_{I(x_0, t)}\log\omega(x - \frac{y}{2^k}, x) dx + \int_{I(x_0, t)}\log\omega(x, x+\frac{y}{2^k}) dx\\
& = \int_{x_0 - \frac{t}{2} - \frac{y}{2^{k+1}}}^{x_0 + \frac{t}{2} - \frac{y}{2^{k+1}}} \log\omega(I(x, \frac{y}{2^k})) dx + \int_{x_0 - \frac{t}{2} + \frac{y}{2^{k+1}}}^{x_0 + \frac{t}{2} + \frac{y}{2^{k+1}}} \log\omega(I(x, \frac{y}{2^k})) dx.
\end{align*}
Then, by rearranging the interval of integration, $2F_k(y)$ is divided into four terms:
\begin{align*}
2F_k(y) & = \left( \int_{x_0 - \frac{t}{2}}^{x_0 - \frac{t}{2} + \frac{y}{2^{k+1}}} \log\omega(I(x, \frac{y}{2^k})) dx - \int_{x_0 - \frac{t}{2} - \frac{y}{2^{k+1}}}^{x_0 - \frac{t}{2}} \log\omega(I(x, \frac{y}{2^k})) dx\right) \\
&\quad+ \left( \int_{x_0 + \frac{t}{2} - \frac{y}{2^{k+1}}}^{x_0 + \frac{t}{2}} \log\omega(I(x, \frac{y}{2^k})) dx - \int_{x_0 + \frac{t}{2}}^{x_0 + \frac{t}{2} + \frac{y}{2^{k+1}}} \log\omega(I(x, \frac{y}{2^k})) dx\right).
\end{align*}
Finally, by the change of the variables again, we have
\begin{align*}
2F_k(y) = \int_{x_0 - \frac{t}{2}}^{x_0 - \frac{t}{2} + \frac{y}{2^{k+1}}} \log\frac{\omega(I(x, \frac{y}{2^k}))}{\omega(x - \frac{y}{2^k}, x)} dx + \int_{x_0 + \frac{t}{2}}^{x_0 + \frac{t}{2} + \frac{y}{2^{k+1}}} \log\frac{\omega(x - \frac{y}{2^k}, x)}{\omega(I(x, \frac{y}{2^k}))}dx.
\end{align*}
By using the statement $(b)$ in Lemma \ref{vanish} as above, we have that
 \begin{equation}\label{log}
 \left|\log\frac{\omega(I(x, \frac{y}{2^k}))}{\omega(x - \frac{y}{2^k}, x)}\right| \leq \log(1+\varepsilon)\leq 4\varepsilon,
 \end{equation}
which implies that
\begin{equation}\label{above}
\left|\frac{F_k(y)}{y} \right|\leq 2^{-(k-1)}\varepsilon. 
\end{equation}

Consequently, we conclude that
$$
|\widehat{A_3}(x_0, t)| = \sum_{k=1}^{\infty}\frac{1}{t}\int_{\frac{t}{2}}^{t}| F_k(y) |\frac{dy}{y}\leq \sum_{k=1}^{\infty}\frac{\varepsilon}{2^{k}}=\varepsilon
$$
for any $x_0 \in \mathbb R$ and any $t \in (0, \delta)$. 
This completes the proof of Claim \ref{c3}. 
\end{proof}

\begin{proof}[Proof of Theorem 1]
Suppose $\omega$ is an $A_\infty$-weight on $\mathbb{R}$ and $\log\omega \in {\rm VMO}(\mathbb R)$. We have $\lim_{N\to\infty}\widehat{A_2}(x_0, t) = 0$ by Claim \ref{c1} and the doubling weight $\omega$ is vanishing on $\mathbb R$ by Remark \ref{examin}. Then $\widehat{A_1}(x_0, t) \to 0$, $\widehat{A_3}(x_0, t) \to 0$ and $A_4(x_0,t)\to 0$ uniformly for $x_0 \in \mathbb R$ as $t \to 0$ by Claims \ref{c2} and \ref{c3} and Proposition \ref{limit}, respectively. Thus, $A(x_0, t) \to 0$ uniformly for $x_0 \in \mathbb R$ as $t \to 0$; namely, the Carleson measure $\eta (z) y^{-1} dxdy$ is vanishing on $\mathbb R_+^2$.

Now we prove the sufficiency. Suppose $\omega$ is an $A_\infty$-weight on $\mathbb{R}$, then $$\lim_{N\to\infty}\widehat{A_2}(x_0, t) = 0$$ by Claim \ref{c1}. Since the doubling weight $\omega$ is vanishing on $\mathbb R$, we have $\widehat{A_1}(x_0, t) \to 0$ and $\widehat{A_3}(x_0, t) \to 0$ uniformly for $x_0 \in \mathbb R$ as $t \to 0$ by Claims \ref{c2} and \ref{c3}, respectively. Combining with the condition that the Carleson measure $\eta (z) y^{-1} dxdy$ is vanishing on $\mathbb R_+^2$, we see $A_4(x_0, t) \to 0$ uniformly for $x_0 \in \mathbb R$ as $t \to 0$. It follows from Proposition \ref{limit} again that $\log \omega \in {\rm VMO}(\mathbb R)$. This completes the proof of Theorem \ref{main}. 
\end{proof}

\end{document}